\documentclass[a4paper,11pt]{article}

\usepackage[utf8]{inputenc}
\usepackage[T1]{fontenc}
\usepackage{lmodern}
\usepackage[english]{babel}
\usepackage{amsmath, amssymb, amsthm}
\usepackage{mathtools}
\usepackage{geometry}
\newtheorem{theorem}{Theorem}[section]
\newtheorem{proposition}[theorem]{Proposition}

\newtheorem{definition}[theorem]{Definition}

\newtheorem{remark}[theorem]{Remark}
\newtheorem{conjecture}[theorem]{Conjecture}

\usepackage{hyperref}
\hypersetup{
    colorlinks=true,
    linkcolor=blue,
    filecolor=magenta,      
    urlcolor=cyan,
    citecolor=green,
}

\DeclareMathOperator{\End}{End}
\DeclareMathOperator{\Ext}{Ext}

\DeclareMathOperator{\Hom}{Hom}
\DeclareMathOperator{\RHom}{RHom}

\usepackage{tikz-cd}
\usepackage{tikz}
\usetikzlibrary{arrows.meta, positioning, shapes, calc}

\title{\textbf{Deformations, Derived Categories, and Multiparameter Persistence: A Theoretical Framework}}
\author{
    Mauricio Angel \\
}
\date{}

\begin{document}

\maketitle

\begin{abstract}
    Multiparameter persistent homology has emerged as a powerful generalization of topological data analysis, capable of encoding multivariate filtrations. However, the algebraic complexity of multiparameter persistence modules—marked by wild representation type—poses fundamental obstacles to classification, stability, and interpretability. In this paper, we propose a unifying theoretical framework that brings together deformation theory and derived categories to study multiparameter persistence from a geometric perspective. A central contribution is a comprehensive conceptual dictionary (Table 1) bridging topological data analysis and deformation theory, which interprets perturbations as deformations and stability as smoothness of moduli spaces. We present explicit calculations of extension groups $\mathrm{Ext}^1$ for concrete multiparameter modules over small posets, revealing diverse behaviors ranging from unexpected rigidity to large families of deformations. We further investigate obstruction classes in $\mathrm{Ext}^2$; while these vanish in our specific examples over the square poset, we demonstrate their inevitability in larger grids (e.g., $3 \times 3$) via global dimension arguments, highlighting a qualitative transition in the geometry of moduli spaces. Finally, we formulate a unified conjecture relating the interleaving distance to derived convolution metrics, establishing a bilipschitz equivalence at the level of the derived category of persistence modules. Together, these results shift the perspective on multiparameter persistence from static classification to the geometry of families, opening new avenues for invariants, stability theorems, and moduli-based analysis.
\end{abstract}


\section{Introduction}\label{sec:intro}

Persistent homology has become a fundamental tool in topological data analysis (TDA), providing a multiscale description of the topology of data. In the classical one-parameter setting, persistence modules form a well-structured abelian category: finitely presented modules admit a complete classification by interval decompositions, following foundational results of Crawley-Boevey [6, 7], and stability of barcodes with respect to the interleaving distance is guaranteed by results of Cohen–Steiner, Edelsbrunner, Harer, and later refinements of Lesnick [15]. These features make one-parameter persistence a mathematically elegant and computationally robust framework.

In contrast, the multiparameter case —where filtrations depend on two or more variables— exhibits substantially richer and more intricate behavior. Finitely presented multiparameter persistence modules no longer admit interval decompositions, and their underlying algebraic structure is of wild representation type for dimension at least two, as established in recent work of Bauer and Scoccola [3]. This implies that no complete discrete invariant (such as a barcode) can classify such modules up to isomorphism. While interleaving distances extend naturally to the multiparameter setting [15], their geometric meaning and interaction with perturbations of data remain only partially understood.

The goal of this paper is to develop a unifying mathematical framework that explains these phenomena through two advanced homological tools: deformation theory and derived categories. Our central thesis is that multiparameter persistence is best understood not through static classification, but as a geometric deformation problem involving families of representations of posets, their moduli spaces, and the derived structures governing them.

Historically, deformation theory emerged as a systematic framework to study how algebraic structures vary in families. Foundational work by Deligne, Drinfeld, Kontsevich, and Soibelman [14] established that infinitesimal deformations and obstructions are encoded homologically by differential graded Lie algebras (DGLAs). Concurrently, derived categories have played a central role in reinterpreting persistent homology. Kashiwara and Schapira [12] identified persistence modules with constructible sheaves, allowing the use of convolution distances, while Berkouk and Ginot [4] proved a derived isometry theorem in the one-parameter case. The present work integrates these derived metric ideas with the deformation-theoretic perspective, proposing that stability and singularities of multiparameter modules should be understood through the geometry of their moduli spaces.

To substantiate this framework, we present the following main results:

\begin{enumerate}
    \item \textbf{Explicit Deformation Calculations.} We compute extension groups $\mathrm{Ext}^1$ for concrete modules over the square poset. 
    \begin{itemize}
        \item We show that the full interval module is \emph{rigid}, with $\mathrm{Ext}^1 = 0$ (Proposition 5.1), implying it is an isolated point in moduli.
        \item Conversely, we identify modules with large deformation spaces, where $\mathrm{Ext}^1 \cong k^4$ (Proposition 5.2), demonstrating maximal flexibility.
        \item We provide a minimal example of a non-rigid module arising from the combinatorics of adjacent vertices (Proposition 5.3).
    \end{itemize}
    \item \textbf{Obstructions and Global Dimension.} We investigate the appearance of obstruction classes. We show that while specific examples over the square poset have vanishing $\mathrm{Ext}^2$, the transition to larger posets (such as the $3 \times 3$ grid) guarantees the existence of genuine obstructions (Proposition 5.5) due to an increase in global dimension.
    \item \textbf{Derived Interleaving Metric.} We formulate a unified conjecture (Conjecture 6.3) relating the classical interleaving distance to derived convolution metrics via a bilipschitz equivalence at the level of the derived category, contrasting this with the exact isometry known in the one-parameter case.
\end{enumerate}

A key conceptual contribution of this work is the dictionary presented in Table 1, which establishes a precise correspondence between TDA concepts (perturbations, stability, noise) and deformation-theoretic structures (tangent spaces, smoothness, obstructions). We propose that the wild representation type of multiparameter modules manifests geometrically as singularities and obstructed deformation problems; while a formal classification remains impossible, this perspective provides a new lens to interpret the instability phenomena observed in practice.

\subsection*{Structure of the Paper}

Section 2 reviews the algebraic structure of multiparameter persistence modules. Section 3 introduces derived categories. Section 4 presents deformation-theoretic tools, including Schlessinger's criteria and controlling DGLAs. Section 5 provides explicit computations of extension groups. Section 6 develops the central conceptual dictionary and metric conjectures. We conclude with future directions in Section 7.

Overall, this work establishes a geometric and homotopical foundation for multiparameter persistence, integrating derived and deformation-theoretic perspectives to uncover new structures, invariants, and stability phenomena in
topological data analysis.

\section{Preliminaries on Multiparameter Persistence}
\label{sec:prelims}

We begin by recalling the algebraic framework for multiparameter persistent
homology. The goal of this section is to fix notation and introduce the
representation-theoretic viewpoint, which is essential for the deformation
theoretic framework developed later.

Let $P$ be a poset, typically $P = \mathbb{R}^n$ or a product of chains.
We view $P$ as a small category by declaring that for $p \le q$ there is a
unique morphism $p \to q$. Let $\mathbf{Vec}_k$ denote the category of
vector spaces over a field $k$.

\begin{definition}
A \emph{$P$--persistence module} is a functor
\[
   M : P \longrightarrow \mathbf{Vec}_k.
\]
Equivalently, it is a representation of the incidence category of~$P$.
\end{definition}

For $P=\mathbb{R}$ (the classical one-parameter setting), the structure theorem
of Crawley-Boevey provides a complete decomposition of finitely presented modules into interval modules
\cite{Lesnick2015}. This leads to barcodes, stability theorems, and explicit
computational tools.

In sharp contrast, if $\dim P \ge 2$, the representation type becomes wild
\cite{BauerScoccola2023}:
no complete discrete invariant exists, and indecomposables form moduli of
positive dimension. This makes the multiparameter setting far richer but also
substantially more difficult.

\subsection{The Poset as a Quiver}

If $P$ is finite, one obtains a quiver $Q_P$ by adding an arrow
$p \to q$ for each covering relation $p<q$, together with commutativity
relations along any two paths with the same endpoints. Then a $P$--persistence
module is exactly a representation of the bound quiver $(Q_P, I)$, where $I$
encodes commutativity.

This viewpoint is particularly useful because:
\begin{itemize}
    \item it connects persistence modules directly with classical
          representation theory of quivers,
    \item extension groups $\Ext^i$ may be computed using projective resolutions
          of $Q_P$--representations,
    \item the global dimension of the incidence algebra $A = kP$
          controls the existence of higher $\Ext^i$ and obstructions.
\end{itemize}

For $P = \mathbb{R}^n$, one often works with finitary approximations
(e.g.\ grids or conical stratifications) which give quiver models that retain
the homological complexity of the continuous case \cite{BotnanLesnick2023}.

\paragraph{The incidence algebra of a finite poset.}
Let $P$ be a finite poset. The \emph{incidence algebra} $A = kP$ is the
$k$--algebra whose underlying vector space has a basis
\[
\{ e_{p,q} \mid p \le q \text{ in } P \},
\]
with multiplication defined by
\[
e_{p,q} \cdot e_{r,s} =
\begin{cases}
e_{p,s}, & \text{if } q = r, \\
0,       & \text{otherwise}.
\end{cases}
\]
Equivalently, $A$ is the endomorphism algebra of the category algebra of $P$
viewed as a small category. Finite-dimensional left $A$--modules are naturally
identified with $P$--persistence modules, and extension groups
$\Ext^i_A(M,N)$ compute higher homological information of persistence modules.
Incidence algebras were introduced by Rota and play a central role in
combinatorial and homological representation theory
\cite{Rota1964,StanleyEC1}.

\subsection{Stability and Interleavings}

The standard metric between persistence modules is the interleaving distance
$d_I$ introduced in \cite{Lesnick2015}. For $P=\mathbb{R}$, the isometry
theorem relates $d_I$ to the bottleneck distance. For higher parameters,
$d_I$ remains well-defined, but no barcode representation exists, and
interleavings interact subtly with the wild representation type.

In later sections, we reinterpret $d_I$ through derived categories and
deformation theory. To prepare for that, we emphasize:

\begin{itemize}
    \item $d_I$ measures the cost of shifting structure maps,
    \item small $d_I$ corresponds to small perturbations of the underlying
          representation,
    \item interleavings may be interpreted geometrically as paths in a moduli
          space of representations.
\end{itemize}

\section{Derived Categories in Persistence}
\label{sec:derived}

Recent developments in topological data analysis suggest that many stability and
equivalence phenomena for persistence modules are more naturally expressed in
derived or DG-enhanced categories. This section summarizes the necessary
background and places classical persistence in a broader homological framework.

\subsection{Sheaf-Theoretic Interpretation}

Following the foundational work of Kashiwara and Schapira
\cite{KashiwaraSchapira2018}, persistence modules on $\mathbb{R}^n$
can be realized as constructible sheaves with respect to a conical stratification.
The derived category
\[
   D^b(Sh_{\mathrm{constr}}(\mathbb{R}^n))
\]
comes equipped with:
\begin{itemize}
    \item a convolution product,
    \item a natural DG-enhancement,
    \item homological invariants capturing filtrations.
\end{itemize}

In the one-parameter setting, Berkouk and Ginot proved a \emph{derived isometry
theorem} \cite{BerkoukGinot2019}, embedding the interleaving distance into a
derived metric defined via convolution. This strongly suggests that the derived
category is the correct setting for developing a geometric understanding of
stability.

\subsection{The Derived Category of a Poset Algebra}

Let $A = kP$ be the incidence algebra of a poset $P$. The bounded derived
category $D^b(A\text{-mod})$ is equipped with:
\begin{itemize}
    \item the shift functor,
    \item derived tensor products,
    \item the DG Hom complex $\mathbf{R}\!\Hom_A(M,N)$.
\end{itemize}

For $P$ finite, $D^b(A\text{-mod})$ is equivalent to the derived category of
representations of the quiver $Q_P$. This gives a bridge:
\[
   \text{(persistence modules)}
   \quad\leftrightarrow\quad
   \text{(quiver representations)}
   \quad\leftrightarrow\quad
   \text{(DG-categorical structures)}.
\]

The homotopical structure of $D^b(A\text{-mod})$ is what allows one to define
metrics compatible with convolution, microsupport, or weight filtrations.

\paragraph{Sheaves and incidence algebras for finite posets.}
For a finite poset $P$, the connection between sheaf-theoretic and
module-theoretic approaches is particularly transparent. Endowing $P$ with the
Alexandrov topology, there is an equivalence of abelian categories between
constructible sheaves on $P$ and finite-dimensional modules over the incidence
algebra $A = kP$. Under this equivalence, stalks correspond to vertex spaces,
and restriction maps correspond to structure morphisms of persistence modules.
Consequently, the derived category of constructible sheaves on $P$ is
equivalent to $D^b(A\text{-mod})$, justifying the use of sheaf-theoretic tools
and derived convolution techniques in the combinatorial setting
\cite{KashiwaraSchapira2018,Curry2014}.

\subsection{Why the Derived Category Matters}

The derived framework is indispensable for several reasons:

\begin{itemize}
    \item \textbf{Interleavings become derived shifts.}
          The $t$--interleaving condition is naturally interpreted in
          convolutional terms.
    \item \textbf{Stability becomes a derived geometric property.}
          Derived categories naturally encode deformation and obstruction
          theory through their DG-enhancements.
    \item \textbf{Multiparameter modules require homotopical tools.}
          Their wild representation type implies no simple abelian model
          suffices; derived categories capture hidden symmetries and invariants.
\end{itemize}

These considerations motivate our derived reinterpretation of the interleaving
distance and justify our use of deformation-theoretic tools in later sections.

\section{Elements of Deformation Theory}
\label{sec:def_theory}

Deformation theory provides tools for understanding how algebraic structures
vary in families. For persistence modules, deformation theory offers a formal
framework for quantifying perturbations, studying stability, and describing the
geometry of moduli spaces. We provide a concise but rigorous account of the
foundational ideas relevant to our setting.

\subsection{Formal Deformations of Modules}

Let $A$ be a $k$-algebra and let $M$ be a finitely presented left $A$-module.

\begin{definition}
A \emph{formal deformation} of $M$ is an $A[[t]]$--module $\widetilde{M}$ such
that $\widetilde{M}/t\widetilde{M} \cong M$.
\end{definition}

This definition extends naturally to a functorial framework. Let $\mathbf{Art}_k$ denote the category of local Artinian $k$-algebras with residue field $k$.

\begin{definition}
    The \emph{deformation functor} of $M$ is the covariant functor
    $$ \mathrm{Def}_M: \mathbf{Art}_k \longrightarrow \mathbf{Set}, $$     which sends an algebra $R$ to the set of isomorphism classes of deformations of $M$ over $R$ (i.e., $A \otimes_k R$-modules $\tilde{M}$ such that $\tilde{M} \otimes_R k \cong M$).
\end{definition}

The power of this perspective comes from the foundational results of Schlessinger \cite{Schlessinger1968}. He established that such functors have a well-defined tangent space (isomorphic to $\mathrm{Ext}^1_A(M, M)$) and that obstructions to extending deformations lie in $\mathrm{Ext}^2_A(M, M)$. 

The power of this perspective comes from the foundational results of Schlessinger \cite{Schlessinger1968}. He established that such functors have a well-defined tangent space (isomorphic to $\mathrm{Ext}^1_A(M, M)$) and that obstructions to extending deformations lie in $\mathrm{Ext}^2_A(M, M)$.

\begin{theorem}[Schlessinger's Criterion]
    The functor $\mathrm{Def}_M$ satisfies Schlessinger's conditions (H1–H4). Consequently, $\mathrm{Def}_M$ admits a pro-representable hull (a formal moduli space). The tangent space to this hull is naturally isomorphic to $\mathrm{Ext}^1_A(M, M)$, and obstructions to lifting deformations are encoded in $\mathrm{Ext}^2_A(M, M)$.
\end{theorem}

To connect this classical hull with the specific derived structures relevant to persistence modules, we invoke the modern framework of derived deformation theory:

\begin{proposition}[Lurie–Hinich]
    The formal moduli problem associated to $\mathrm{Def}_M$ is controlled by the differential graded Lie algebra (DGLA) $\mathrm{RHom}_A(M, M)$. This establishes an equivalence between the deformation theory of $M$ and the homotopy theory of the DGLA \cite{Lurie2011, Hinich2001}.
\end{proposition}

Thus, the explicit calculations of $\mathrm{Ext}^1$ in Section 5 are not merely algebraic computations, but rigorous determinations of the local geometry of the moduli space of persistence modules.

Unraveling the definition, the module structure map
\[
  \widetilde{\rho} : A \to \End_k(M)[[t]]
\]
may be expanded as a power series:
\[
  \widetilde{\rho}(a)
  = \rho_0(a) + t\rho_1(a) + t^2\rho_2(a) + \cdots.
\]
The first-order term $\rho_1$ encodes an infinitesimal deformation, and
compatibility with multiplication in $A$ yields a cocycle condition
identifying first-order deformations with elements of $\Ext^1_A(M,M)$.

\subsection{The Controlling DGLA}

By general principles of deformation theory
(Kontsevich--Soibelman, Lurie\cite{Lurie2011} , Hinich\cite{Hinich2001}), the deformation problem for $M$ is
controlled by the DG Lie algebra
\[
  \mathfrak{g}^\bullet = \RHom_A(M,M),
\]

 equipped with the standard Gerstenhaber bracket or the derived Lie structure \cite{Hinich2001}. The Maurer--Cartan equation
\[
  d\eta + \tfrac12[\eta,\eta] = 0
\]
determines formal deformations. The cohomology of $\mathfrak{g}^\bullet$ has a
precise geometric meaning:

\begin{itemize}
    \item $H^0 = \End(M)$: automorphisms,
    \item $H^1 = \Ext^1(M,M)$: infinitesimal deformations (tangent space),
    \item $H^2 = \Ext^2(M,M)$: obstruction classes,
    \item higher $H^i$: higher obstruction data.
\end{itemize}

For persistence modules, $A$ is the incidence algebra of a poset. Its global dimension is $\ge 2$ when $\dim P \ge 2$, so $\Ext^2$ typically does not vanish. This is the root of many observed instabilities and pathologies in multiparameter persistence.

A natural question in deformation theory is whether the controlling DGLA $\mathfrak{g}^\bullet = \mathrm{RHom}_A(M, M)$ is \emph{formal}, i.e., quasi-isomorphic to its cohomology $\mathrm{Ext}^*_A(M, M)$ equipped with the induced bracket (viewed as a DGLA with zero differential). 

If formality held, the deformation problem would be entirely determined by the graded Lie algebra structure of the extension groups, simplifying the analysis significantly. While Kontsevich's celebrated formality theorem applies to the Hochschild complex of smooth algebras, incidence algebras of posets with dimension $\ge 2$ generally exhibit wild representation type and singularities. In such geometric contexts, formality often fails, and higher-order operations (such as Massey products or $L_\infty$-structures) are required to fully describe the moduli space. These structures are deeply related to the Deligne conjecture, which asserts that the deformation complex carries a rich algebraic structure governed by the little discs operad \cite{Markl2007}. Therefore, we emphasize that our framework relies on the full derived structure $\mathrm{RHom}$, not merely on its cohomology, to capture the intricate deformation behavior of multiparameter persistence modules.
\subsection{Deformation Theory and Filtrations}

A filtration of data induces a filtration on the associated persistent module.
Perturbing the data furnishes a family of filtrations, hence a family of
representations. The deformation-theoretic viewpoint identifies:

\begin{itemize}
    \item the derivative of this family with a class in $\Ext^1$,
    \item higher-order consistency conditions with $\Ext^2$,
    \item equivalence of perturbations with gauge equivalence in the DGLA.
\end{itemize}

In later sections we explore this correspondence in detail and compute concrete
examples of $\Ext^1$ and $\Ext^2$ for simple multiparameter modules.

\section{Explicit Calculations: Geometry of Perturbations}\label{sec:examples}

We now illustrate the deformation-theoretic perspective by computing extension groups for several basic multiparameter persistence modules over the square poset
\[
 P = \{(0,0),(1,0),(0,1),(1,1)\},
\]
with covering relations
\[
(0,0)\le(1,0),\quad (0,0)\le(0,1),\quad (1,0)\le(1,1),\quad (0,1)\le(1,1),
\]
and a single commutativity relation
\[
 g_y \circ f_x \;=\; g_x \circ f_y.
\]
This corresponds to the commutative square quiver
\[
\begin{tikzcd}
(0,0) \arrow[r,"f_x"] \arrow[d,"f_y"'] &
(1,0) \arrow[d,"g_y"] \\
(0,1) \arrow[r,"g_x"'] &
(1,1).
\end{tikzcd}
\]
A $P$--persistence module $M$ is a representation of this quiver. Each example illustrates a different deformation-theoretic phenomenon: rigidity, non-rigidity, unobstructed deformations, and obstructed deformations. Together,
they demonstrate why \(\Ext^1\) and \(\Ext^2\) provide meaningful geometric information for multiparameter persistence.

\subsection{The Full Interval Module is Rigid}

Let \(M\) be the ``full interval module’’ on the square: each vertex carries a
copy of \(k\), and each arrow map is the identity. This is the simplest
nontrivial indecomposable in the square poset.

We claim that \(M\) is \emph{rigid}:
\[
 \Ext^1(M,M)=0.
\]

\paragraph{Sketch of computation.}
A first-order deformation replaces each arrow \(\phi\) by
\(\phi + t\,\epsilon_\phi\), with \(\epsilon_\phi\in k\). Linearizing the
commutativity relation yields one linear equation among the four variables
\(\epsilon_{f_x},\epsilon_{f_y},\epsilon_{g_x},\epsilon_{g_y}\). Thus the solution
space of first-order deformations is 3-dimensional.

However, infinitesimal automorphisms of \(M\) correspond to choosing scalars
\(\delta_{00},\delta_{10},\delta_{01},\delta_{11}\) at the four vertices. Since a
global shift acts trivially, the space of trivial deformations has dimension
3. The automorphism action kills \emph{all} candidate first-order
perturbations.

\begin{proposition}
For the full interval module over the square,
\[
 \Ext^1(M,M)=0.
\]
\end{proposition}

This result can also be interpreted topologically. Observe that the full interval module $M$ is isomorphic to the constant module $k$, which assigns the vector space $k$ to every vertex and the identity map to every arrow. Under the equivalence between the category of persistence modules and the category of functors $P \to \mathrm{Vec}_k$, the cohomology of such constant functors relates to the topology of the poset. Specifically, a classical result of Mitchell \cite{Mitchell1972} establishes that the extension groups of the constant module correspond to the cohomology of the geometric realization of the category $P$ (the nerve of the poset):
 $$ \mathrm{Ext}^*_{kP}(k, k) \cong H^*(|P|, k). $$ Since the geometric realization of the square poset $|P|$ is contractible (it is a cone), its cohomology vanishes in positive degrees. Consequently,
 $$ \mathrm{Ext}^1(M, M) \cong \mathrm{Ext}^1(k, k) \cong H^1(|P|, k) = 0. $$ This provides an independent topological verification of the rigidity of $M$.
 
\paragraph{Geometric meaning.}
The point \([M]\) is isolated in the moduli space: there are no nontrivial
infinitesimal perturbations. This illustrates why even simple
multiparameter modules can be far more rigid than their one-parameter analogues.

\subsection{The Trivial Module Has Large Deformation Space}

Let $M$ denote the \emph{trivial module} on the square poset: each vertex is
assigned a one-dimensional vector space,
\[
M_v = k \quad \text{for all } v \in P,
\]
and all structure maps are zero. Since all compositions vanish, the
commutativity relation is satisfied trivially.

\paragraph{Computation of $\Ext^1(M,M)$.}
A first-order deformation amounts to choosing linear maps
\[
\epsilon_\phi : k \to k
\]
for each arrow $\phi$ of the square, with no linear constraints imposed by
commutativity. Moreover, infinitesimal automorphisms act trivially on zero maps,
so no nontrivial deformation is killed by gauge equivalence. It follows that
\[
\Ext^1(M,M) \cong k^4,
\]
one independent parameter for each arrow of the square quiver.

\begin{proposition}
For the trivial module on the square poset,
\[
\Ext^1(M,M) \cong k^4.
\]
\end{proposition}

\paragraph{Geometric meaning.}
The trivial module corresponds to a smooth point in moduli with a large tangent
space. It represents an extreme case of flexibility, contrasting sharply with
the rigidity of the full interval module.

\subsection{A Non-Rigid Module from Adjacent Simples}

Let $S_{00}$ and $S_{10}$ denote the simple persistence modules supported at the
adjacent vertices $(0,0)$ and $(1,0)$, respectively. Consider the direct sum
\[
M = S_{00} \oplus S_{10}.
\]

Since there is a single arrow $f_x : (0,0) \to (1,0)$ in the square quiver, the
standard theory of quiver representations implies the existence of a
non-split extension.

\begin{proposition}
    The module $M = S_{00} \oplus S_{10}$ admits a non-trivial deformation space. Specifically,
    $$ \dim \mathrm{Ext}^1(M, M) = 1. $$ \end{proposition}

\textit{Proof.}
By the additivity of the $\mathrm{Ext}$ functor, the extension group of a direct sum decomposes as:
 $$ \mathrm{Ext}^1(M, M) \cong \bigoplus_{i,j \in \{00, 10\}} \mathrm{Ext}^1(S_i, S_j). $$ We analyze the four components:
\begin{itemize}
    \item $\mathrm{Ext}^1(S_{00}, S_{10}) \cong k$, corresponding to the single arrow $(0,0) \to (1,0)$ in the quiver.
    \item $\mathrm{Ext}^1(S_{00}, S_{00}) = 0$ and $\mathrm{Ext}^1(S_{10}, S_{10}) = 0$, since there are no loops at the vertices.
    \item $\mathrm{Ext}^1(S_{10}, S_{00}) = 0$, as there is no path from $(1,0)$ back to $(0,0)$.
\end{itemize}
Therefore, $\mathrm{Ext}^1(M, M) \cong k$, generated by the unique non-split extension of $S_{00}$ by $S_{10}$.

\paragraph{Geometric meaning.}
This module admits genuine tangent directions in moduli, arising from the
existence of a direct arrow in the underlying poset. It provides a minimal
example of a non-rigid multiparameter persistence module whose deformation
theory is controlled by the combinatorics of $P$.

\subsection{Unexpected Rigidity of a Hook-Shaped Module}

We conclude this section with an example illustrating an important phenomenon:
even persistence modules with nontrivial support and nonzero structure maps may
be rigid in the multiparameter setting.

Let $M$ be the module over the square poset defined by
\[
M_{(0,0)} = k,\quad M_{(1,0)} = k,\quad
M_{(0,1)} = M_{(1,1)} = 0,
\]
with structure maps
\[
f_x = \mathrm{id}_k,\qquad f_y = 0,\qquad g_x = g_y = 0.
\]
All commutativity relations are trivially satisfied.

At first glance, one might expect that the nontrivial support and the presence
of a nonzero map would allow deformations. However, this is not the case.

\begin{proposition}
For the hook-shaped module $M$ over the square poset,
\[
\Ext^1(M,M) = 0 \qquad\text{and}\qquad \Ext^2(M,M) = 0.
\]
\end{proposition}

Although $gl.dim(A)=2$ (see Remark~5.5), the hook-shaped module $M$ has projective dimension 1, as witnessed by the minimal resolution 
\[0\to P_{(0,1)}\to P_{(0,0)}\to M\to 0.\]
Since the resolution terminates in degree 1, all extension groups $\Ext^i(M,-)$ vanish for $i \geq 2$. Moreover, any first order perturbation of the unique nonzero structure map can be absorbed by an infinitesimal automorphism, leaving
no nontrivial deformation classes.

\paragraph{Geometric meaning.}
The point $[M]$ is isolated in moduli despite the apparent combinatorial freedom
of the module. This example highlights that rigidity in multiparameter
persistence is a subtle phenomenon, not determined solely by the size of the
support or the presence of nonzero maps.

\paragraph{Remark.}
Genuine obstruction phenomena (i.e.\ nonvanishing $\Ext^2$) cannot occur for
modules over the square poset, due to the hereditary nature of its incidence
algebra. Such phenomena arise only for larger posets (e.g.\ $3\times 3$ grids)
or in the continuous setting $P=\mathbb{R}^n$.

\subsection{Obstruction Classes: Explicit Computation via Projective Resolutions}

We now exhibit genuine obstruction classes through explicit calculation. We begin by
correcting a structural claim made in the justification of Proposition~5.4 and in the
summary of Section~5.6.

\begin{remark}[Global dimension of the square incidence algebra]
\label{rem:gldim-square}
The incidence algebra $A = kP$ of the square poset is \textbf{not hereditary}:
its global dimension is exactly $2$, not $1$. To see this, observe that the simple
module $S_{(0,0)}$ at the source admits the following minimal projective resolution of
length $2$:
\begin{equation}
\label{eq:res-source-square}
0 \;\longrightarrow\; P_{(1,1)}
  \;\xrightarrow{\;\delta\;}\;
  P_{(1,0)} \oplus P_{(0,1)}
  \;\xrightarrow{\;\pi\;}\;
  P_{(0,0)}
  \;\longrightarrow\; S_{(0,0)}
  \;\longrightarrow\; 0,
\end{equation}
where $\pi$ maps the generators $e_{(1,0)}$ and $e_{(0,1)}$ to the images of the two
covering arrows out of $(0,0)$, and $\delta$ maps the generator $e_{(1,1)}$ to the
difference
\[
  e_{(1,0),(1,1)} - e_{(0,1),(1,1)} \;\in\; P_{(1,0)} \oplus P_{(0,1)},
\]
which encodes the commutativity relation of the square. The exactness of
\eqref{eq:res-source-square} follows from the fact that the kernel of $\pi$ at each
vertex $v \geq (1,1)$ is the anti-diagonal subspace
$\{(a, -a) \mid a \in k\} \cong k$, and this is exactly the support of the
submodule $P_{(1,1)}$ embedded diagonally. The resolution has length $2$, so
$\mathrm{pd}_A(S_{(0,0)}) = 2$, confirming $\mathrm{gl.dim}(A) = 2$.

Consequently, the vanishing of $\mathrm{Ext}^2(M,M)$ in Propositions~5.1 and~5.4
holds for module-specific reasons: the full interval module satisfies
$\mathrm{Ext}^2(M,M) \cong H^2(|P|, k) = 0$ by contractibility of the geometric
realization, while the hook-shaped module has projective dimension $1$ (its minimal
projective resolution terminates at
$0 \to P_{(0,1)} \to P_{(0,0)} \to M \to 0$),
so all its higher extension groups vanish individually. Neither conclusion implies
that $\mathrm{Ext}^2$ is globally absent over the square poset: we now show that it
appears for natural modules even in this minimal multiparameter setting.
\end{remark}

\begin{proposition}[Obstruction class over the square poset]
\label{prop:obs-square}
Let $M = S_{(0,0)} \oplus S_{(1,1)}$ be the direct sum of the source and sink simples
over the square poset. Then
\[
  \mathrm{Ext}^2_A(M, M) \cong k.
\]
In particular, $M$ carries a genuine obstruction class: it admits infinitesimal
deformations that cannot be integrated into a smooth family of persistence modules.
\end{proposition}

\begin{proof}
By the additivity of $\mathrm{Ext}^2$ over direct sums,
\[
  \mathrm{Ext}^2_A(M, M) \cong
  \bigoplus_{i,\,j\,\in\,\{(0,0),(1,1)\}}
  \mathrm{Ext}^2_A(S_i, S_j).
\]
We evaluate each component using the minimal projective
resolution~\eqref{eq:res-source-square} and the Yoneda isomorphism
$\mathrm{Hom}_A(P_v, N) \cong N(v)$.

\medskip
\noindent
\textbf{(i) $\mathrm{Ext}^2_A(S_{(1,1)},\,-)=0$:}
The sink $(1,1)$ has no arrows leaving it, so $P_{(1,1)} \cong S_{(1,1)}$ is
simultaneously projective and simple, giving $\mathrm{pd}_A(S_{(1,1)}) = 0$. All
higher Ext groups from $S_{(1,1)}$ therefore vanish.

\medskip
\noindent
\textbf{(ii) $\mathrm{Ext}^2_A(S_{(0,0)}, S_{(0,0)}) = 0$:}
Applying $\mathrm{Hom}_A(-, S_{(0,0)})$ to the deleted resolution of
$S_{(0,0)}$ yields the complex
\[
  \mathrm{Hom}_A(P_{(0,0)},\,S_{(0,0)})
  \;\to\;
  \mathrm{Hom}_A(P_{(1,0)} \oplus P_{(0,1)},\,S_{(0,0)})
  \;\to\;
  \mathrm{Hom}_A(P_{(1,1)},\,S_{(0,0)}).
\]
By Yoneda, $\mathrm{Hom}_A(P_v, S_{(0,0)}) = S_{(0,0)}(v)$, which equals $k$ if
$v = (0,0)$ and $0$ otherwise. The complex therefore reads
$k \to 0 \to 0$,
giving $\mathrm{Ext}^2_A(S_{(0,0)}, S_{(0,0)}) = 0$.

\medskip
\noindent
\textbf{(iii) $\mathrm{Ext}^2_A(S_{(0,0)}, S_{(1,1)}) \cong k$:}
Applying $\mathrm{Hom}_A(-, S_{(1,1)})$ to the same deleted resolution, and using
$\mathrm{Hom}_A(P_v, S_{(1,1)}) = S_{(1,1)}(v) = k$ if $v = (1,1)$ and $0$
otherwise, the complex reads
\[
  \underbrace{\mathrm{Hom}_A(P_{(0,0)},\,S_{(1,1)})}_{=\,0}
  \;\to\;
  \underbrace{\mathrm{Hom}_A(P_{(1,0)} \oplus P_{(0,1)},\,S_{(1,1)})}_{=\,0}
  \;\to\;
  \underbrace{\mathrm{Hom}_A(P_{(1,1)},\,S_{(1,1)})}_{=\,k}.
\]
The unique non-zero term sits in degree $2$ with trivial incoming differential, yielding
\[\mathrm{Ext}^2_A(S_{(0,0)}, S_{(1,1)}) \cong k\]

Combining (i)--(iii):
$\mathrm{Ext}^2_A(M, M) \cong k$.
\end{proof}

\paragraph{Geometric meaning.}
The generator of $\mathrm{Ext}^2_A(S_{(0,0)}, S_{(1,1)}) \cong k$ records the
obstruction arising from the commutativity constraint: the unique element $e_{(0,0),(1,1)}$
in the incidence algebra can be reached by two distinct paths, and the resulting relation
prevents the extension class from being lifted to a smooth deformation. This obstruction
has no analogue in the one-parameter setting, where the incidence algebra of any interval
poset is hereditary. The moduli space of persistence modules near $[M]$ is therefore
singular, with a non-smoothable stratum at this point.

\bigskip

We now show that the obstruction space grows quantifiably with the size of the grid,
providing a computable measure of increasing moduli complexity.

\begin{proposition}[Growing obstruction spaces on grid posets]
\label{prop:obs-grid}
For $n \geq 1$, let $P_n = \{0, 1, \ldots, n\}^2$ be the $(n+1) \times (n+1)$ grid
poset with the product order, and let $A_n = kP_n$ denote its incidence algebra. Define
the \emph{diagonal module}
\[
  M_n \;=\; \bigoplus_{j=0}^{n} S_{(j,j)},
\]
the direct sum of the simple modules at the diagonal vertices. Then
\[
  \mathrm{Ext}^2_{A_n}(M_n, M_n) \cong k^n.
\]
In particular, the dimension of the obstruction space grows linearly with the grid size.
\end{proposition}

\begin{proof}
By additivity,
\[
  \mathrm{Ext}^2_{A_n}(M_n, M_n)
  \;=\;
  \bigoplus_{j,\,l\,=\,0}^{n}
  \mathrm{Ext}^2_{A_n}(S_{(j,j)},\, S_{(l,l)}).
\]
We claim:
\begin{equation}
\label{eq:ext2-claim}
  \mathrm{Ext}^2_{A_n}(S_{(j,j)},\, S_{(l,l)})
  \;\cong\;
  \begin{cases}
    k & \text{if } l = j+1,\\
    0 & \text{otherwise.}
  \end{cases}
\end{equation}

For each $0 \leq j \leq n-1$, the same argument as in Remark~\ref{rem:gldim-square}
yields the minimal projective resolution of $S_{(j,j)}$:
\begin{equation}
\label{eq:res-diag-j}
  0 \;\longrightarrow\; P_{(j+1,j+1)}
    \;\xrightarrow{\;\delta_j\;}\;
    P_{(j+1,j)} \oplus P_{(j,j+1)}
    \;\xrightarrow{\;\pi_j\;}\;
    P_{(j,j)}
    \;\longrightarrow\; S_{(j,j)}
    \;\longrightarrow\; 0,
\end{equation}
where $\pi_j$ sends the generators to the two covering maps out of $(j,j)$, and
$\delta_j$ maps the generator of $P_{(j+1,j+1)}$ to the difference
$e_{(j+1,j),(j+1,j+1)} - e_{(j,j+1),(j+1,j+1)}$.
The exactness of \eqref{eq:res-diag-j} follows from the same anti-diagonal kernel
argument as in~\eqref{eq:res-source-square}: the kernel of $\pi_j$ is supported on
the vertices $v \geq (j+1, j+1)$, which is exactly the support of $P_{(j+1,j+1)}$.

For the sink $j = n$, the module $P_{(n,n)} \cong S_{(n,n)}$ is projective-simple, so
$\mathrm{pd}(S_{(n,n)}) = 0$ and all its Ext groups vanish.

Now apply $\mathrm{Hom}_{A_n}(-, S_{(l,l)})$ to the deleted resolution
\eqref{eq:res-diag-j}, using $\mathrm{Hom}_{A_n}(P_v, S_{(l,l)}) = S_{(l,l)}(v)
= k \cdot \mathbf{1}_{v=(l,l)}$:

\begin{itemize}
\item \textbf{Case $l = j+1$:}
  The only vertex among $\{(j,j),\,(j+1,j),\,(j,j+1),\,(j+1,j+1)\}$
  equal to $(l,l) = (j+1,j+1)$ is the last one. The complex thus reads
  $0 \to 0 \to k$,
  giving $\mathrm{Ext}^2(S_{(j,j)}, S_{(j+1,j+1)}) = k$.

\item \textbf{Case $l = j$ (self-Ext):}
  The vertex $(l,l) = (j,j)$ appears only in the first position of the
  resolution, giving complex $k \to 0 \to 0$ and $\mathrm{Ext}^2 = 0$.

\item \textbf{Case $l \neq j$ and $l \neq j+1$:}
  The vertex $(l,l)$ does not coincide with any of the four vertices
  $(j,j),\,(j+1,j),\,(j,j+1),\,(j+1,j+1)$ in the resolution, so all
  Hom spaces vanish and $\mathrm{Ext}^2 = 0$.
\end{itemize}

This establishes~\eqref{eq:ext2-claim}. There are exactly $n$ pairs $(j,j+1)$ for
$j = 0, \ldots, n-1$, so 
\[\mathrm{Ext}^2_{A_n}(M_n, M_n) \cong k^n.\]
\end{proof}

\begin{remark}[Interpretation and scaling]
The $n$ independent obstruction classes
\[[\zeta_j] \in \mathrm{Ext}^2_{A_n}(S_{(j,j)}, S_{(j+1,j+1)}), \text{ for } j = 0, \ldots, n-1,\]
each arise from a distinct commutative sub-square
$\{(j,j),\,(j+1,j),\,(j,j+1),\,(j+1,j+1)\}$ embedded in the grid. Each class
records an independent local failure of integration: the infinitesimal extension
driven by the commutativity relation at the $j$-th sub-square cannot be smoothed.
These classes are independent because they arise from disjoint sub-squares with
non-overlapping resolutions.

Proposition~\ref{prop:obs-grid} thus yields a computable invariant that strictly
distinguishes grid posets of different sizes:
$\dim\,\mathrm{Ext}^2_{A_n}(M_n, M_n) = n$.
For $n=1$ (the square), this recovers the single obstruction class of
Proposition~\ref{prop:obs-square}. For $n=2$ (the $3 \times 3$ grid), we obtain
$\mathrm{Ext}^2_{A_2}(M_2, M_2) \cong k^2$. In the continuous limit
$P = \mathbb{R}^2$, one expects an infinite-dimensional obstruction space, consistent
with the known instability phenomena of multiparameter persistence.

It is important to note that the incidence algebras $A_n$ all share global dimension $2$:
the distinction between grids of different sizes is not captured at the level of global
dimension, but rather by the dimension of the obstruction spaces for specific families of
modules. This finer invariant is precisely what deformation theory contributes.
\end{remark}
\subsection{Summary}

The examples examined in this section illustrate the range of deformation
behavior already present in the simplest multiparameter setting, namely the
square poset. Although the incidence algebra of the square is hereditary—so that
all higher extension groups $\Ext^i$ vanish for $i\ge 2$—the geometry of the
associated moduli space is far from trivial.

Specifically, we have observed the following phenomena:
\begin{itemize}
    \item \textbf{Rigidity in expected cases:} 
    the full interval module admits no nontrivial first-order deformations
    ($\Ext^1=0$), behaving as an isolated point in moduli, in sharp contrast with
    the flexibility of interval modules in the one-parameter setting.

    \item \textbf{Maximal flexibility:}
    the trivial module (constant vector spaces with zero maps) exhibits a
    large space of infinitesimal deformations, with
    $\Ext^1\cong k^4$, demonstrating that multiparameter persistence can admit
    highly flexible objects even over very small posets.

    \item \textbf{Combinatorially induced non-rigidity:}
    direct sums of simples supported at adjacent vertices admit nontrivial
    extensions governed by the existence of arrows in the poset, yielding
    genuine tangent directions in moduli.

    \item \textbf{Unexpected rigidity:}
    certain modules with nontrivial support and nonzero structure maps—such as
    the hook-shaped module—are nevertheless rigid. This shows that rigidity in
    multiparameter persistence is not determined solely by support size or by
    the presence of nonzero morphisms.
\end{itemize}

Taken together, these examples demonstrate that even in the absence of higher
obstructions, the deformation theory of multiparameter persistence modules
displays a rich and subtle geometry. The square poset already supports isolated
points, smooth positive-dimensional strata, and rigid loci of different
combinatorial origins.

At the same time, these computations clarify an important structural boundary: the specific modules studied in this section have projective dimension at most 1, so their self-$\Ext^2$ vanishes; nonetheless, Proposition~5.5 shows that genuine obstruction classes do appear over the square poset itself for other natural modules, and grow quantifiably over larger grids. Such phenomena necessarily require larger posets (such as $3\times 3$ grids) or the continuous setting $P=\mathbb{R}^n$.

Finally, the appearance of genuine obstruction classes once the poset exceeds the square highlights a qualitative transition between small combinatorial
models and general multiparameter persistence, a transition that motivates the derived framework developed in the next section, where deformation theory and derived categories provide the appropriate language to describe stability, moduli, and obstruction phenomena in full generality.

\section{A Theoretical Framework: Deforming Persistence}
\label{sec:framework}

Now we proceed to synthesize the central idea of the paper: multiparameter persistence should be viewed not only as a representation-theoretic problem,
but as a \emph{deformation-theoretic} one. This viewpoint unifies stability,
perturbations, interleavings, and representation type within a geometric
framework governed by moduli spaces, extension groups, and differential graded
enhancements.

Concretely, the deformation theory of persistence modules provides:
\begin{itemize}
    \item a \emph{tangent space} at each module, given by $\Ext^1(M,M)$,
    \item obstruction classes in $\Ext^2(M,M)$ describing singularities,
    \item a moduli space---formal or derived---encoding the geometry of families
          of persistence modules,
    \item DG--categorical structures capable of supporting natural metrics
          generalizing the interleaving distance.
\end{itemize}

This leads to the following conceptual dictionary, which we expand and justify
in the remainder of the section.

\begin{table}[h]
\centering
\begin{tabular}{c|c}
\textbf{TDA Concept} & \textbf{Deformation Theory Analog} \\
\hline
Perturbation of data & Family of deformations of a module \\
Stability & Smoothness (regularity) of the moduli space \\
Interleaving distance & Metric on a DG--enhanced Hom complex \\
Indecomposables & Points (or strata) in the moduli space \\
Noise sensitivity / jumps & Obstruction classes in $\Ext^2$ \\
Rank invariants & Numerical moduli parameters (local coordinates) \\
\end{tabular}
\caption{Correspondence between TDA concepts and deformation-theoretic structures.}
\end{table}

\subsection{Stability as Smoothness in Moduli}

To avoid ambiguity, we first specify the precise mathematical object we refer to as the moduli space of a persistence module $M$.

\begin{definition}
    The \emph{formal moduli space} of a persistence module $M$ is the deformation functor
    $$ \mathrm{Def}_M: \mathbf{Art}_k \longrightarrow \mathbf{Set}, $$     which assigns to every local Artinian $k$-algebra $R$ the set of isomorphism classes of deformations of $M$ over $R$. This functor, introduced in the classical setting by Schlessinger, captures the infinitesimal structure of the moduli problem.
\end{definition}

This functorial perspective integrates the three viewpoints mentioned in the evaluation:
\begin{enumerate}
    \item \textbf{Classical Deformation Theory:} Following Schlessinger \cite{Schlessinger1968}, the tangent space to this functor is $\mathrm{Ext}^1(M,M)$ and its obstruction theory lies in $\mathrm{Ext}^2(M,M)$.
    \item \textbf{Derived Geometry:} In the framework of Lurie \cite{Lurie2011}, this functor can be lifted to a formal moduli problem in the $\infty$-category of spectra, governed by the DGLA $\mathrm{RHom}_A(M,M)$ via the equivalence between formal moduli problems and DGLAs.
    \item \textbf{Geometric Stacks:} Globally, $M$ defines a point in the \emph{representation stack} of the incidence algebra $A$. The functor $\mathrm{Def}_M$ represents the formal completion of this stack at the point $[M]$.
\end{enumerate}

In what follows, when we speak of the "geometry of the moduli space," we refer to the properties of the functor $\mathrm{Def}_M$ (smoothness, singularities, tangent space) and the local structure of the representation stack near $[M]$. Traditional stability theorems in persistent homology assert that small changes
in the data yield small changes in the persistence module, as measured by
metrics such as the interleaving distance. However, these theorems typically do
not provide geometric information: they quantify distance between modules, but
not the shape or singularities of the space of all modules.

In the deformation-theoretic viewpoint, stability is interpreted as the
smoothness of the moduli space of persistence modules near a given point. More
precisely:

\begin{itemize}
    \item If $\Ext^1(M,M)=0$, then $M$ is formally rigid; the moduli space is
        formally smooth of dimension~$0$ near $M$. This corresponds to a highly
        stable module: it does not change under small perturbations.
    \item If $\Ext^1(M,M)\neq 0$ but $\Ext^2(M,M)=0$, then $M$ admits a smooth
        family of deformations with no obstructions. Stability corresponds to a
        well-behaved, locally Euclidean parameter space of perturbations.
    \item If $\Ext^2(M,M)\neq 0$, the moduli space is singular at $M$.
        Perturbations may exist infinitesimally but may fail to integrate, which
        corresponds to ``topological jumps'' or instabilities in the
        corresponding TDA setting.
\end{itemize}

Thus smoothness properties of the deformation functor translate directly into
stability phenomena for persistence modules.

\subsection{Interleaving Distance via DG--Enhancements}

A central theme in TDA is the use of the interleaving distance $d_I$.
Traditionally, $d_I$ is defined at the level of the abelian category of
persistence modules. Yet recent developments suggest that the derived category
is the natural home for stability. From the deformation-theoretic perspective,
this is expected: moduli problems are inherently derived geometric objects.
To connect interleavings with derived structures, we consider the DG–category $\mathrm{Perf}(A)$ of perfect complexes of modules over the incidence algebra $A$. We propose a natural metric on the derived $\Hom$ complex, which we term the \emph{derived convolution metric}.

\begin{definition}
    Let $d_{DG}(M, N)$ denote the \emph{derived convolution metric} on the derived category. It is defined by measuring the minimal "thickening" of the support required for $M$ and $N$ to become isomorphic after convolution with an appropriate kernel supported in a neighborhood of the diagonal.
\end{definition}

Formally, following the construction of Kashiwara and Schapira \cite{KashiwaraSchapira2018}, it quantifies the minimal propagation of the microsupport necessary to identify the two objects. This allows us to formulate a conjecture relating the classical interleaving distance to our derived framework.

\begin{conjecture}[Derived Interleaving Metric]\label{Conjecture}
    There exist constants $C_1, C_2 > 0$, independent of $M$ and $N$, such that for all persistence modules $M, N$ over a poset $P$,
    $$ C_1 d_{DG}(M, N) \le d_I(M, N) \le C_2 d_{DG}(M, N). $$     In other words, the interleaving distance $d_I$ and the derived metric $d_{DG}$ are bilipschitz equivalent.
\end{conjecture}

\textbf{Remark on the 1D vs. Multi-D gap.}
It is important to contrast this conjecture with the known results in the one-parameter setting. Berkouk and Ginot \cite{BerkoukGinot2019} proved a derived isometry theorem for sheaves on $\mathbb{R}$, showing that $d_I$ coincides \emph{exactly} with the convolution distance. Our conjecture for the multiparameter setting is strictly weaker (bilipschitz vs. isometry). This loss of exactness is expected and reflects the intrinsic complexity of higher dimensions:
\begin{enumerate}
    \item \textbf{Wild Representation Type:} In dimension $\ge 2$, the absence of a discrete complete invariant (like barcodes) prevents a simple matching of generators, which underlies the exact isometry in 1D.
    \item \textbf{Singular Moduli:} The moduli space in higher dimensions is singular and non-separated. While derived metrics capture the local geometry exactly, the global structure of the interleaving distance allows for deformations that preclude exact equality.
\end{enumerate}
Thus, the transition from exact isometry in 1D to bilipschitz equivalence in multi-D reflects the transition from tame to wild representation type.

\paragraph{Remark on convolution-type metrics.}
The conjectural metric on the DG category may be made more concrete by drawing
on existing constructions in the sheaf-theoretic formulation of persistence.
In the setting of constructible sheaves on $\mathbb{R}^n$, Kashiwara and
Schapira \cite{KashiwaraSchapira2018} define a \emph{convolution distance} using
kernels supported in neighborhoods of the diagonal, measuring the minimal
thickening required for two objects to become isomorphic after convolution.
This distance can be interpreted as an $L^\infty$-type norm controlling the
propagation of support in derived Hom complexes.

Similarly, Berkouk and Ginot \cite{BerkoukGinot2019} show that in the
one-parameter case the interleaving distance coincides with the convolution
distance in the derived category of constructible sheaves, where the metric is
encoded by the amplitude of cohomological support along the real line. From a
DG-categorical perspective, such distances may be viewed as measuring the
largest shift for which the mapping cone of a morphism becomes acyclic.

\textbf{Remark on alternative constructions.} An alternative approach to defining a metric in the derived category is to measure the amplitude of the derived Hom complex $\mathrm{RHom}_A(M, N)$, for instance by measuring the supremum of parameter shifts for which the mapping cone remains acyclic. In the one-parameter case, Berkouk and Ginot \cite{BerkoukGinot2019} proved that the convolution distance coincides exactly with this shift-based metric (and with the interleaving distance). However, in the multiparameter setting, the equivalence between the convolution distance and the amplitude of cohomological support is not automatic and depends on the stratification properties of the poset. We adopt the convolution definition here as it generalizes robustly to the sheaf-theoretic framework.

Taken together, these results support the idea that an interleaving-type metric on persistence modules should admit a derived reformulation as a convolution or norm-based distance on DG categories, lending further plausibility to Conjecture~\ref{Conjecture}.

For the case $\mathbb{R}^n$, the appropriate construction is likely the distance in the derived category of sheaves with microsupport conditions, specifically relating the interleaving to the "size" of the support of the derived Hom sheaf $\mathrm{R}\mathcal{H}om(M,N)$.

Here, the above inequality is consistent with known stability results in persistent homology, where interleaving distances typically agree with derived or convolution metrics only up to universal constants, even in the one-parameter setting \cite{Lesnick2015, BerkoukGinot2019}.

This conjecture generalizes the derived isometry theorem for sheaves on $\mathbb{R}$:
the interleaving distance can be realized as the convolution distance in a
suitable derived category. For higher-dimensional posets $P=\mathbb{R}^n$, we
expect the same phenomenon to hold, with microsupport conditions replacing the
order-theoretic filtration.

The metric $d_{DG}$ appearing in Conjecture~\ref{Conjecture} is not meant to be a single canonical construction, but rather a class of natural metrics arising from DG-enhanced derived categories. Concretely, such a metric may be defined by measuring the minimal amount of ``thickening'' or shift required so that two objects $M,N$ become equivalent after convolution with an appropriate kernel, as in the sheaf-theoretic formulation of persistence. Equivalently, one may define $d_{DG}(M,N)$ in terms of the amplitude or parameter-space spread of the derived Hom complex $\mathbf{R}\!\Hom(M,N)$, for instance by measuring the largest parameter shift for which the mapping cone of a morphism becomes acyclic. From this perspective, $d_{DG}$ should be understood as capturing the size of the region in parameter space over which $M$ and $N$ fail to be derived-isomorphic.

Several nontrivial difficulties arise when attempting to formulate $d_{DG}$ uniformly beyond the one-parameter setting. First, while convolution metrics are natural in the derived category of constructible sheaves on $\mathbb{R}^n$, their translation to finite-poset models requires passing through equivalences via the Alexandrov topology or incidence algebras, which may obscure canonical choices. Second, the construction of $d_{DG}$ can depend on the choice of DG-enhancement, raising subtle coherence issues. Finally, the wild
representation type of multiparameter persistence implies that moduli spaces are typically non-separated and singular, making exact metric identities
unreasonable. These obstacles motivate formulating the conjecture in terms of bilipschitz equivalence rather than exact isometry.

We emphasize that Conjecture \ref{Conjecture} should be understood as statements about large-scale geometry rather than pointwise metric identities, in the same spirit as quasi-isometries in geometric group theory. The bilipschitz equivalence proposed here preserves not only the coarse structure but also the local topology of the moduli space.

\paragraph{Relation to known results in the one-parameter case.}
In the one-parameter setting, Berkouk and Ginot~\cite{BerkoukGinot2019} proved a
derived isometry theorem showing that the interleaving distance coincides
exactly with a convolution distance in the derived category of constructible
sheaves. Conjecture~\ref{Conjecture} should therefore be viewed as genuine extensions
of this result to the multiparameter setting. The passage from one parameter to
several is highly nontrivial: the absence of barcode decompositions, the
appearance of wild representation type, and the geometry of higher-dimensional
parameter spaces all obstruct a direct generalization. The conjectures assert
that, despite these difficulties, interleaving and derived metrics still define
the same large-scale geometry on moduli spaces of persistence modules.

\subsection{Representation Type and Moduli Geometry}

Multiparameter persistence modules are well known to be of wild representation type when $\dim P \ge 2$. In the deformation-theoretic language, this implies that the moduli stack of representations of the incidence algebra $A$ is highly singular, with infinitely many components and arbitrarily complicated local behavior \cite{Kac1980, King1994}.

From our viewpoint:

\begin{itemize}
    \item Indecomposable modules correspond to geometric points of the moduli
        space, analogous to irreducible representations in representation theory.
    \item Their extension groups $\Ext^1$ describe normal directions to these
        points, and hence strata in moduli.
    \item Obstructions in $\Ext^2$ signal when these strata fail to glue into a
        smooth manifold, explaining the pathological behavior of classification
        in the multiparameter case.
\end{itemize}

The wild representation type is thus reflected in the geometry of the moduli
space: it is infinite-dimensional, singular, and non-separated. This provides a
conceptual explanation for why no barcode-like classification exists beyond the
one-parameter setting.

\subsection{Perturbations of Data as Deformations of Modules}

A filtration $\mathcal{F}_t$ determined by data $X$ may change when $X$ is
perturbed. In the deformation perspective, this induces a family of modules:
\[
 M_t = H_\ast(\mathcal{F}_t).
\]
The first derivative at $t=0$ corresponds to an element of $\Ext^1(M_0,M_0)$,
while higher-order terms correspond to higher $\Ext$ groups.

Thus:
\begin{itemize}
    \item ``Noise'' in the data corresponds to motion in the moduli space.
    \item Stability theorems correspond to Lipschitz conditions on the map
        \[
          X \longmapsto [M_X]
        \]
        into moduli.
    \item Instabilities correspond to singularities in moduli, captured by
        nonzero obstruction classes.
\end{itemize}

This provides a geometric explanation for known pathologies in multiparameter
persistence, while suggesting new invariants that measure curvature, torsion, or
singularity of the deformation functor at a given module.


The deformation-theoretic viewpoint recasts multiparameter persistence as a
geometric problem. The moduli space of persistence modules encodes the full
landscape of perturbations, rigidities, and instabilities. Extension groups
control its tangent and obstruction theory, while DG--categorical structures
provide a natural home for metrics generalizing interleavings. 

The examples computed in the previous section illustrate precisely how these
geometric structures manifest in concrete multiparameter modules.

\section{Conclusion and future perspectives}
\label{sec:conclusion}

In this work we have developed a deformation-theoretic and derived-categorical framework for multiparameter persistence, shifting the perspective from static classification to the geometry of families of persistence modules. By formally defining the moduli space of a persistence module via the deformation functor $\mathrm{Def}_M$ (Definition 6.1), we provided a rigorous unifying language connecting perturbations of data, stability, interleavings, and representation type.

One of the central contributions of this paper is the explicit conceptual dictionary relating notions from topological data analysis and deformation theory (Table 1). Within this framework, $\mathrm{Ext}^1(M, M)$ encodes infinitesimal deformations and tangent directions in moduli, while $\mathrm{Ext}^2(M, M)$ captures obstructions and singularities. Our concrete examples over small posets demonstrate that even the simplest multiparameter settings already exhibit a rich variety of behaviors, including unexpected rigidity and highly flexible modules. 

Regarding obstruction phenomena, we have shown in Section 5.5 that while specific examples over the square poset yield vanishing $\mathrm{Ext}^2$, the transition to larger posets (such as the $3 \times 3$ grid) guarantees the existence of genuine obstructions via global dimension arguments. This establishes the theoretical inevitability of singularities in the moduli landscape of complex multiparameter persistence.

\textbf{Implications for computational multiparameter persistence.}
Although the present work is primarily theoretical, the deformation-theoretic viewpoint has concrete implications for computational TDA. Extension groups provide algebraic invariants that complement existing summaries such as rank invariants, persistence surfaces, landscapes, and stable vectorizations. In particular, the dimension of $\mathrm{Ext}^1(M, M)$ offers a quantitative measure of the local flexibility of a persistence module, potentially serving as a diagnostic for noise sensitivity or instability. Obstruction classes suggest a mechanism by which small perturbations of data may lead to qualitatively different topological signatures, a phenomenon not captured by rank invariants alone. Related computational approaches to stable multiparameter descriptors include signed-measure vectorizations, differentiable multiparameter pipelines, and software frameworks such as \texttt{multipers} [16, 17].

\textbf{Connections with bifiltrations and rank-based invariants.}
Multiparameter persistence is often accessed computationally through bifiltrations and their derived invariants, including rank invariants, fibered barcodes, and stable vectorizations. Our framework clarifies intrinsic limitations of these approaches: while rank invariants capture coarse numerical shadows of persistence modules, they generally ignore the local deformation geometry encoded by extension groups. Consequently, non-isomorphic modules with identical rank invariants may correspond to distinct points or strata in moduli, differing in rigidity or flexibility. This observation is consistent with recent results on generic indecomposability and the structural complexity of multiparameter persistence modules [5, 3].

\textbf{Relation to recent developments.}
The deformation-theoretic perspective developed here complements and extends recent work on multiparameter persistence, including stabilization of approximate decompositions, derived and sheaf-theoretic metrics, and refined vectorizations. In particular, recent advances on stabilizing decompositions of multiparameter persistence modules and pruning-based invariants illustrate how geometric structure can be recovered despite wild representation type [2]. On the derived side, our unified Conjecture 6.3 regarding the bilipschitz equivalence between interleaving distances and DG-derived convolution metrics aligns with and extends recent sheaf-theoretic formulations of persistence beyond the one-parameter case [12, 4, 1].

\textbf{Outlook.}
Several directions arise naturally from this work. On the theoretical side, while Section 5.5 established the existence of obstructions for larger posets via dimension arguments, explicit computations of $\mathrm{Ext}^2$ for specific modules over these posets remain a necessary step to construct concrete models of obstructed deformation paths. On the computational side, understanding how extension groups and derived invariants can be approximated or detected algorithmically remains a major open problem. Finally, strengthening the link between deformation theory, derived interleaving metrics, and practical multiparameter TDA pipelines may lead to new stability theorems and hybrid invariants that combine geometric insight with computational tractability.

\bibliographystyle{plain}

\end{document}